\documentclass[12pt,a4paper]{amsart}
\usepackage{url}
\usepackage{amsmath,amsthm,amsfonts,amssymb,latexsym}
\newtheorem{theorem}{Theorem}[section]

\newtheorem{lemma}[theorem]{Lemma}

\newtheorem{corollary}[theorem]{Corollary}

\theoremstyle{definition}

\newtheorem{question}[theorem]{Question}

\newcommand{\U}{\mathcal U}
\newcommand{\w}{\omega}

\newcommand{\IQ}{\mathbb Q}

\newcommand{\IP}{\mathbb P}

\newcommand{\K}{\mathcal{K}}

\newcommand{\G}{\mathcal{G}}
\newcommand{\HH}{\mathcal{H}}

\newcommand{\Z}{\mathcal{Z}}

\newcommand{\F}{\mathcal{F}}

\newcommand{\V}{\mathcal{V}}

\newcommand{\uhr}{\upharpoonright}

\newcommand{\name}[1]{\dot{#1}}
\newcommand{\la}{\langle}
\newcommand{\ra}{\rangle}

\newcommand{\forces}{\Vdash}

\newcommand{\hot}{\mathfrak}
\newcommand{\sfF}{\mathsf F}
\newcommand{\sfH}{\mathsf H}

\newcommand{\nothing}[1]{}

\newcommand{\ve}{\varepsilon}

\title[Coideals and combinatorial covering properties]{Coideals as remainders of groups distinguishing between combinatorial covering properties}

\author[G. Molica Bisci, 
D.D. Repov\v{s}$^*$ 
 and 
 L. Zdomskyy]{Giovanni Molica Bisci, 
Du\v{s}an D. Repov\v{s}$^*$ 
 and 
 Lyubomyr
Zdomskyy}

\address{Dipartimento di Scienze Pure e Applicate (DiSPeA), Universit\`{a} degli
Studi di Urbino Carlo Bo, Piazza della Repubblica 13, 61029 Urbino (Pesaro e Urbino), Italy}
 \email{giovanni.molicabisci@uniurb.it}
\urladdr{http://www.giovannimolicabisci.it/}

\address{Faculty of Education, and Faculty of Mathematics and Physics, University
of Ljubljana
\& Institute of Mathematics, Physics and Mechanics,
1000 Ljubljana, Slovenia} \email{dusan.repovs@guest.arnes.si}
\urladdr{http://repovs.splet.arnes.si/}

\address{Institute of Mathematics (Kurt G\"odel Research Center),
University of Vienna, Kolingasse 14-16, A-1090 Vienna, Austria.}
\email{lzdomsky@gmail.com}
\urladdr{http://www.logic.univie.ac.at/\~{}lzdomsky/}

\thanks{The second author was supported by the
Slovenian Research Agency grants P1-0292, N1-0278, N1-0114, N1-0083, J1-4031, and J1-4001. The
third author would like to thank the Austrian Science Fund FWF
(grant I 3709) for generous support for this research. We thank the referee for comments and suggestions.\\
{\it Corresponding author:} Du\v{s}an D. Repov\v{s}, University
of Ljubljana \& Institute of Mathematics, Physics and Mechanics. Email: {\it dusan.repovs@guest.arnes.si}}


\begin{document}
\begin{abstract}
In this paper we construct consistent examples of subgroups of
$2^\w$ with Menger remainders which fail to have other stronger
combinatorial covering properties. This answers several open
questions asked by Bella,  Tokgoz and Zdomskyy (Arch. Math. Logic
\textbf{55} (2016), 767--784).
\end{abstract}
\maketitle

\section{Introduction}
All topological spaces  will be assumed to be Tychonoff.  For a
space $X$ and its compactification $bX$ the complement $bX\setminus
X$ is called a \emph{remainder} of $X$.

The study of the interplay between covering properties of Tychonoff
spaces and those of their remainders takes its origin in the seminal
work  of Henriksen and Isbell \cite{HenIsb57}.
 Being highly homogeneous objects,
topological groups restrict the properties of their remainders in a
special, strong way, unlike general topological spaces. The study of
this phenomenon in our context was started by
 Arhangel'skii  \cite{Arh08, Arh09} and further pursued in his
 joint works with  Choban, van Mill, and others, see \cite{ArhCho21, ArhMil13, ArhMil14,ArhMil15} and
 the references therein.
 
This motivated the study in \cite{BelTokZdo16} of topological groups
whose remainders have combinatorial covering properties
 which lie between the $\sigma$-compactness and
Lindel\"of property. Recall from \cite{COC1} that a space $X$ is
said to be \emph{Menger} (or has the \emph{Menger property}) if  for
any sequence $\la \U_n: n\in\w\ra$ of open covers of $X$ one can
pick finite sets $\mathcal V_n\subset \mathcal U_n$ in such a way
that $\{\bigcup\mathcal V_n:n\in\omega\}$ is a cover of $X$. A
family $\{W_n:n\in\w\}$ of subsets of $X$ is called an
\emph{$\w$-cover} of $X$, if for every $F\in [X]^{<\w}$, the  set
$\{n\in\w:F\subset W_n\}$ is infinite. The property of
\emph{Scheepers} is defined in the same way as the Menger property,
the only difference being that we additionally demand that
$\{\bigcup\mathcal V_n:n\in\omega\}$ is a $\w$-cover of $X$. It is
immediate that
$$\sigma\mbox{-compact} \Rightarrow
\mbox{Scheepers} \Rightarrow \mbox{Menger} \Rightarrow
\mbox{Lindel\"of}.$$ Through a sequence of reductions it was proved
in \cite{BelTokZdo16} that there exists a Scheepers ultrafilter  if
and only if there exists a topological group $G$ such that $\beta
G\setminus G$ is Scheepers and not $\sigma$-compact if and only if
there exists a topological group  $G$ such that all finite powers of
$\beta G\setminus G$ are Menger   and are not $\sigma$-compact.
Here,  $\mathcal P(\w)$ is as usually identified with the Cantor
space $2^\w$ via characteristic functions, and subsets of $\mathcal
P(\w)$ are considered with the subspace topology. Thus the existence
of a topological group $G$ such that $\beta G\setminus G$ is
Scheepers (resp. has all finite powers Menger) and not
$\sigma$-compact is independent from ZFC: Such a group exists under
$\mathfrak d=\mathfrak c$, and its existence yields $P$-points, see
\cite{BelTokZdo16} and the references therein. Furthermore, it was
proved in \cite{BelTokZdo16} that the existence of a topological
group $G$ such that $(\beta G\setminus G)^2$ is
 Menger but not $\sigma$-compact  is independent from ZFC as well.

Since the same approach does not allow to solve the question whether
consistently every Menger remainder of a topological group is
$\sigma$-compact\footnote{This question remains open.}, it was asked in \cite{BelTokZdo16} whether the Scheepers and
Menger properties can be distinguished by remainders of topological
groups at all, and whether a Menger remainder of a topological group
can have a non-Menger square. In case of a negative answer outright
in ZFC this would give the consistency of all Menger remainders of
topological groups being $\sigma$-compact, simply by applying the
aforementioned results. However, we provide here two alternative
proofs that both of these questions have consistently
 the affirmative answer, by constructing counterexamples using different approaches,
  see Theorems~\ref{0main}, \ref{0main_f}, and
 Corollary~\ref{cor01}. Our topological groups are actually non-meager $P$-filters
 on $\w$, hence metrizable, $0$-dimensional,  totally bounded, and hereditarily Baire by \cite{Mar98}.
 Since the Menger property is preserved by finite products of
 metrizable spaces in the Miller model \cite{Zdo18} and coincides
 with the Scheepers property under $\hot u<\hot g$ by \cite{Zdo05}
 (this inequality holds
in the Miller model by \cite[Theorem~2]{BlaLaf89} combined with the
results of \cite{BlaShe89}),
 filters like in Theorems~\ref{0main} and \ref{0main_f}
 cannot be obtained in ZFC. Theorem~\ref{0main_density_z} is a
 variation of Theorem~\ref{0main_f} motivated by the question
 whether the density one filter can be diagonalized by a (proper)
 poset adding no dominating reals.

As discussed above, Theorem~\ref{0main} (stating that there exists a
filter $\F$ such that among other properties, both $\F$ and $\F^+$
are Menger, and hence $\F$ cannot be meager by
\cite[Prop.~2]{Zdo05}), cannot be proved in ZFC. In
section~\ref{imp_res} we investigate how far the assumptions on $\F$
can be weakened so that it is still impossible to get such a filter
in ZFC. In this context let us recall  that there are Menger
non-meager filters in ZFC, see \cite{RepZdoZha14}. On the other
hand, the existence of a filter $\F$ in ZFC such that $\F^+$ is
Menger is  unknown and constructing such a filter without additional
set-theoretic assumptions (if it is possible at all) would be
extremely difficult: If $\F^+$ is Menger then $\F$ is non-meager and
$P$, see Corollary~\ref{obs_obv}, and it is a famous
 open problem to construct a non-meager $P$-filter in ZFC. In Theorem~\ref{0_anti_main}
we show that consistently there are no Menger filters $\F$ such that
$\F^+$ is Menger as well.

Let us note that  the properties of Menger, Scheepers,  having
Menger square, etc., are preserved by perfect maps in both
directions.  This implies that if one of the remainders of a space
$X$ has one of these covering properties, then all other also have
it, see the beginning of \cite[\S~3]{BelTokZdo16} for more detailed
explanations.

All undefined topological notions can be found in \cite{Eng}. For
the definitions and basic properties of cardinal characteristics
used in this paper we refer the reader to \cite{Bla10}.

\section{The Main Result}

We need to recall some standard  as well as introduce some ad-hoc
notation and terminology. A family $\HH\subset[\w]^\w$ is called a
\emph{semifilter} if for every $H\in\HH,$ $n\in\w$, and $X\supset
H\setminus n$ we have $X\in\HH$. For a set $\F\subset \mathcal
P(\w)$ and $n\in\w$ we denote by $\la\F\ra_n$ the set
$\{\bigcap\F':\F'\in [\F]^{\leq n}\}$, and $\la\F\ra$ stands for
$$\{X\subset\w:\exists n\exists Y\in\la\F\ra_n (Y\setminus n\subset X)\}.$$
$\F$ is said to be \emph{centered} if $\la\F\ra_n\subset [\w]^\w$
for all $n\in\w$, i.e., if the intersection of any finite subfamily
of $\F$ is infinite. In this case, $\la\F\ra$ is the smallest free
filter on $\w$ containing $\F$. Let us note that  if $\F$ is compact
then so is $\la\F\ra_n$ for any $n\in\w$, being a continuous image
of $\F^n$.  Similarly, by  $\la\F\ra_s$ we shall denote the smallest
semifilter containing $\F$, i.e.,
$$\la\F\ra_s=\{X\subset\w\: :\: \exists F\in\F\exists n\in\w(F\setminus n\subset X)\}.$$
For a family $\sfF$ of subsets of $\mathcal P(\w)$ (i.e.,
$\sfF\subset\mathcal P(\mathcal P(\w))$) we define the
\emph{$\la\ra_*$-saturation} of $\sfF$ as the smallest subfamily
$\sfF_1\supset\sfF$ of $\mathcal P(\mathcal P(\w))$ such that
$\la\bigcup\sfF'\ra_n \in\sfF_1$ for any $\sfF'\in [\sfF_1]^{<\w}$
and $n\in\w$. It is  easy to write the $\la\ra_*$-saturation of a
family $\sfF$ in a precise way,
 thus proving that the $\la\ra_*$-saturation of an infinite $\sfF$ has the same cardinality as $\sfF$,
 and it consists of compact subsets of $\la\bigcup\sfF\ra$ if each $\F\in\sfF$ is compact.
We say that $\sfF$ is \emph{$\la\ra_*$-saturated} if it coincides
with its $\la\ra_*$-saturation.

Given families $\F,\HH$ of subsets of $\w$, we denote by
$\F\wedge\HH$ the family $\{F\cap H:F\in\F, H\in\HH\}$. Again, if
$\F,\HH$ are compact, then so is $\F\wedge\HH$. For
$\sfF,\sfH\subset\mathcal P(\w)$ we define the
\emph{$(\sfF,\wedge)$-saturation} of $\sfH$ as the smallest
subfamily $\sfH_1\supset\sfH$ of $\mathcal P(\mathcal P(\w))$ such
that $\F\wedge\HH\in\sfH_1$ for any $\F\in\sfF$ and $\HH\in\sfH_1$.
Again, it is straightforward to write the $(\sfF,\wedge)$-saturation
of $\sfH$ in the precise way, thus proving that it is a subfamily of
$\mathcal P(\la\bigcup\sfF\ra\wedge\bigcup\sfH)$ of size at most
$\max\{\w,|\sfF|,|\sfH|\}$, and it consists of compact sets if so do
$\sfF$ and $\sfH$. We call $\sfH$ to be
\emph{$(\sfF,\wedge)$-saturated} if it coincides with its
$(\sfF,\wedge)$-saturation.

The next easy auxiliary fact is very similar to \cite[Prop.~7]{GuzHruMar17}
and can be  derived from the latter one in a rather straightforward way.
However, for reader's convenience we give a direct proof. 

\begin{lemma}\label{bidi}
Let $\sfF$ be a family of compact subsets of $[\w]^\w$ of size
$|\sfF|<\min\{\hot d,\hot r\}$ and $\phi:\w\to\w$  a finite-to-one 
surjection. Then there exists  $Z\subset\w$ such that
$$|\phi^{-1}[Z]\cap F|=|(\w\setminus \phi^{-1}[Z])\cap F|=\w$$
for all $F\in\bigcup\sfF$.
\end{lemma}
\begin{proof}
Let us first assume that $\phi$ is monotone and consider
the strictly increasing number sequence
$\la k_n:n\in\w\ra$   with $k_0=0$ such that
$\phi^{-1}(n)=[k_n,k_{n+1})$ for all $n\in\w$. 

For every $\F\in\sfF$ let $h_\F:\w\to\w$ be a strictly increasing function such that
$[h_\F(n), h_{\F}(n+1))\cap F\neq\emptyset$ for all $F\in\F$ and $n\in\w$. Since $|\sfF|<\hot d$, there
exists a strictly increasing $h\in\w^\w$ such that
$$ I_\F:=\{ n\in\w\: : \: (|[h(n),h(n+1))\cap h_\F[\w]|\geq 2)\}$$
is infinite  for all $\F\in\sfF$. Without loss of generality we may assume that
$h[\w]\subset \{k_n:n\in\w\}$, and hence for every (infinite) $I\subset\w$ there exists (an infinite) $Z\subset\w$
such that $\bigcup_{n\in I}[h(n),h(n+1))=\phi^{-1}[Z]$.

 Since $|\F|<\hot r$ there exists $I\subset\w$ such that
 $|I\cap I_\F|=|I_\F\setminus I|=\w$ for all $\F\in\sfF$.
 We claim that
 $$ \big|\bigcup_{n\in I}[h(n),h(n+1))\cap F\big|=\big|\bigcup_{n\in \w\setminus I}[h(n),h(n+1))\cap F\big|=\w $$
 for all $F\in\bigcup\sfF$. Indeed, let us find $\F\in\sfF$ containing $F$ and  pick $n\in I\cap I_\F$.
 Then there exists $m\in\w$ such that
 $[h(n),h(n+1))\supset [h_\F(m),h_\F(m+1))$,
 and hence
 $$\emptyset\neq F\cap [h_\F(m),h_\F(m+1))\subset F\cap [h(n),h(n+1))\subset F\cap\bigcup_{n\in I}[h(n),h(n+1))\cap F.$$
The case of $\w\setminus I$ is analogous.
\smallskip

Now suppose that $\phi$ is arbitrary
 finite-to-one surjection  from $\w$ to $\w$.
Fix a bijection $\theta:\w\to\w$ such that
$\phi\circ\theta$ is a monotone surjection. It follows from the above that there exists
$Z\subset\w$ such that
$$\big|(\phi\circ\theta)^{-1}[Z]\cap \theta^{-1}[F]\big|=\big|(\w\setminus (\phi\circ\theta)^{-1}[Z])\cap \theta^{-1}[F]\big|=\w$$
for all $F\in\bigcup\sfF$, i.e.,
$$\big|\theta^{-1}[\phi^{-1}[Z]]\cap \theta^{-1}[F]\big|=\big|(\w\setminus \theta^{-1}[\phi^{-1}[Z]])\cap \theta^{-1}[F]\big|=\w$$
and thus also
$$|\phi^{-1}[Z]\cap F|=|(\w\setminus \phi^{-1}[Z])\cap F|=\w$$
for all $F\in\bigcup\sfF$ because $\theta$ is a bijection.
\end{proof}

A semifilter $\F$ is called a \emph{$P$-semifilter} if for every
sequence $\la F_n:n\in\w\ra$ of elements of $\F$ there exists a
sequence $\la K_n:n\in\w\ra$ such that $K_n\in [F_n]^{<\w}$ for all
$n\in\w$ and $\bigcup_{n\in\w}K_n\in\F$. If
$\F^+=\{X\subset\w:\forall F\in\F(X\cap F\neq\emptyset)\}$ is a
$P$-semifilter, then we also say that $\F$ is a
\emph{$P^+$-semifilter}. Each semifilter $\F$ on $\w$ gives rise to
the semifilter $\F^{(<\w)}$ on $[\w]^{<\w}\setminus\{\emptyset\}$
generated by the family $\{[F]^{<\w}\setminus\{\emptyset\}\: :\:
F\in\F \}$. If $\F$ is a filter we shall call $\F^+$ the \emph{coideal} of $\F$.

Next, we put together several  known facts about Menger semifilters
established in \cite{ChoRepZdo15} and \cite{GuzHruMar17} and get a
potentially useful characterization.

\begin{theorem}\label{basics}
Let $\F$ be a semifilter on $\w$. Then the following statements are
equivalent:
\begin{itemize}
\item[$(1)$]  $\F$ is Menger;
\item[$(2)$] For every sequence $\la\K_n:n\in\w\ra$ of compact subsets of $\F^+$
there exists an increasing sequence $\la m_n:n\in\w\ra\in\w^\w$ with
the property $\bigcup_{n\in\w}(K_n\cap m_n)\in\F^+$ for any $\la
K_n:n\in\w\ra\in\prod_{n\in\w}\K_n$;
\item[$(3)$] For every sequence $\la\K_n:n\in\w\ra$ of compact subsets of $\F^+$
there exists an increasing sequence $\la m_n:n\in\w\ra\in\w^\w$ with
the property $\bigcup_{n\in\w}\big(K_n\cap
[m_{n-1},m_n)\big)\in\F^+$ for\footnote{We set here $m_{-1}=0$.} any
$\la K_n:n\in\w\ra\in\prod_{n\in\w}\K_n$; and
\item[$(4)$]  $\F^{(<\w)}$ is a $P^+$-semifilter.
\end{itemize}
\end{theorem}
\begin{proof}
The equivalence $(1)\Leftrightarrow (4)$ was established in
\cite[Claim~2.4]{ChoRepZdo15}, and its proof works verbatim also for
semifilters. The implication $(1)\to (3)$ was obtained in
\cite[Prop.~3.4]{ChoRepZdo15}, its proof again works for semifilters
without any changes, while $(3)\to (2)$ is straightforward. Thus it
remains to prove $(2)\to (4)$. Let $\la E_n:n\in\w\ra$ be a sequence
of elements of $(\F^{(<\w)})^+$. For each $n$ set
$\K_n=\{X\subset\w:\forall e\in E_n (X\cap e\neq\emptyset)\}$ and
note that $\K_n\subset\F^+$: If $K\cap F=\emptyset$ for some $K\in
\K_n$ and $F\in\F$, then there is no $e\in E_n\cap [F]^{<\w}$, which
contradicts our choice of $E_n$. Let $\la m_n:n\in\w\ra$ be as in
$(2)$. We claim that $\bigcup_{n\in\w} E_n\cap\mathcal
P(m_n)\in(\F^{(<\w)})^+$, which will imply $(4)$. Indeed, otherwise
there exists $F\in\F$ such that $e\setminus F\neq\emptyset$ for each
$e\in E_n\cap\mathcal P(m_n)$ and $n\in\w$. Thus
$$K_n:=\bigcup\{e\setminus F:e\in E_n, e\subset m_n\}\cup(\w\setminus m_n)\in\K_n$$
for all $n$. However,
\begin{eqnarray*}
F\cap \bigcup_{n\in\w}(K_n\cap m_n)=F\cap
\bigcup_{n\in\w}\bigcup\{e\setminus F:e\in E_n, e\subset m_n\}=\\
=\bigcup_{n\in\w} (F\cap \bigcup\{e\setminus F:e\in E_n, e\subset
m_n\})=\emptyset,
\end{eqnarray*}
and hence $\bigcup\{K_n\cap m_n:n\in\w\}\not\in\F^+$, which
contradicts $(2)$.
\end{proof}

\begin{corollary} \label{obs_obv}
Each Menger semifilter is $P^+.$
\end{corollary}

Theorem~\ref{basics} will be crucial for the proof of the following
fact, which is the main result of the paper.

\begin{theorem}\label{0main}
($\hot r=\hot d=\hot c$). There exists a  filter $\G$ on $\w$ such that
\begin{itemize}
\item[$(1)$] $\G$ and $\G^+$ are Menger;
\item[$(2)$] For every finite-to-one surjection $\phi:\w\to\w$ there exists $X\subset\w$ such that
$\phi^{-1}[X],\phi^{-1}[\w\setminus X]\in\G^+$.
\end{itemize}
\end{theorem}
\begin{proof}
Let us fix the following enumerations:
\begin{itemize}
\item $\{\la\K^\alpha_n:n\in\w\ra:\alpha<\hot c\}$=: the family of all sequences of compact subsets of
$[\w]^\w$;
\item $\{\phi_\alpha:\alpha<\hot c\}$=: the family of all finite-to-one surjections from $\w\to\w$.
\end{itemize}
By recursion over $\alpha<\hot c$ we shall construct a sequence
$\la\la \sfF_\alpha,\sfH_\alpha\ra:\alpha<\hot c\ra$ such that
\begin{itemize}
\item[$(a)$] $\sfF_\alpha,\sfH_\alpha$ are families of compact subsets of $[\w]^\w$
of size $<\hot c$   with $\sfF_0=\sfH_0=\{\{\w\}\}$;
\item[$(b)$] $\bigcup\sfF_\alpha$ is centered and $\sfF_\alpha$ is $\la\ra_*$-saturated;
\item[$(c)$] $\sfF_\alpha\subset\sfF_{\alpha'}\cap\sfH_{\alpha'}$ and $\sfH_\alpha\subset\sfH_{\alpha'}$
 for any $\alpha\leq\alpha'$;
\item[$(d)$] $\bigcup\sfH_\alpha\subset \la\bigcup\sfF_\alpha\ra^+$,
and $\sfH_\alpha$ is $(\sfF_\alpha,\wedge)$-saturated;
\item[$(e)$] If
$\bigcup_{n\in\w}\K^\alpha_n \subset \la\bigcup\sfF_\alpha\ra$,
 then there exists an increasing
number sequence
$\la m^\alpha_n:n\in\w \ra$ such that
$\sfF_{\alpha+1}\ni \K_\alpha$,
where $\K_\alpha=\big\{\bigcup_{n\in\w}(K_n\cap m^\alpha_n):\la K_n:n\in\w\ra\in\prod_{n\in\w}\K^\alpha_n \big\};$
\item[$(f)$] If
$\bigcup_{n\in\w}\K^\alpha_n \not\subset\la\bigcup\sfF_\alpha\ra$,
 then there exists $K_\alpha\in \bigcup_{n\in\w}\K^\alpha_n$ such that
 $\{\w\setminus K_\alpha\}\in\sfH_{\alpha+1}$;
 \item[$(g)$]
 If
$\bigcup_{n\in\w}\K^\alpha_n \subset\la\bigcup\sfH_\alpha\ra_s$,
then there exists an increasing
number sequence
$\la l^\alpha_n:n\in\w \ra$ such that
$\sfH_{\alpha+1}\ni \mathcal L_\alpha$,
where $\mathcal L_\alpha=\big\{\bigcup_{n\in\w}(K_n\cap l^\alpha_n):\la K_n:n\in\w\ra\in\prod_{n\in\w}\K^\alpha_n \big\};$
\item[$(h)$]
 If
$\bigcup_{n\in\w}\K^\alpha_n \not\subset\la\bigcup\sfH_\alpha\ra_s$,
 then there exists $L_\alpha\in \bigcup_{n\in\w}\K^\alpha_n$ such that
 $\{\w\setminus L_\alpha\}\in\sfF_{\alpha+1}$;
 \item[$(j)$] There exists  $Z_\alpha\subset\w$
 such that $\{\phi_\alpha^{-1}[Z_\alpha],\phi_\alpha^{-1}[\w\setminus Z_\alpha]\}\in\sfH_{\alpha+1}$.
\end{itemize}
First, let us assume that we have constructed a sequence $\la\la
\sfF_\alpha,\sfH_\alpha\ra:\alpha<\hot c\ra$ satisfying $(a)$-$(j)$.
We claim that
$$\G:=\big\la\bigcup\{\bigcup\sfF_\alpha:\alpha<\hot c\}\big\ra$$ is as required. It follows that $\G^+$ equals $$\G_1:=\big\la\bigcup\{\bigcup\sfH_\alpha:\alpha<\hot c\}\big\ra_s.$$
Indeed, $\G_1\subset\G^+$
follows directly from $(d)$. To prove that $\G^+\subset\G_1$
let us fix any
$X\not\in\G_1$. Let $\alpha<\hot c$ be such that  $\K^\alpha_n=\{X\}$ for every $n\in\w$
and note that $\la\K^\alpha_n:n\in\w\ra$ satisfies the premises of $(h)$. Therefore there exists
$L_\alpha\in\bigcup_{n\in\w}\K^\alpha_n=\{X\}$ (i.e., $L_\alpha=X$) such that $\{\w\setminus L_\alpha\}\in\sfF_{\alpha+1}$,
and hence $\w\setminus X\in\G$ which yields $X\not\in\G^+$.

Next, we shall establish that both $\G$ and $\G^+$ are Menger.
Let  $\la\K_n:n\in\w\ra$ be a sequence of compact subsets of $\G$ and
$\alpha$ be such that $\K^\alpha=\K^\alpha_n$ for all $n\in\w$. It follows that
$\bigcup_{n\in\w}\K^\alpha_n\subset\la\bigcup\sfF_\alpha\ra$. Indeed, otherwise
by $(f)$ there exists $K_\alpha\in\bigcup_{n\in\w}\K^\alpha_n$ such that $\{\w\setminus K_\alpha\}\in\sfH_{\alpha+1}$,
and therefore $\w\setminus K_\alpha\in\G^+$, which contradicts $K_\alpha\in\G$. Thus
$\la\K^\alpha_n:n\in\w\ra$ fulfills the premises of $(e)$, which yields
 an increasing sequence $\la m_n:n\in\w\ra\in\w^\w$
with the property $\bigcup_{n\in\w}(K_n\cap m_n)\in\G$ for any $\la
K_n:n\in\w\ra\in\prod_{n\in\w}\K^\alpha_n.$ Applying
Theorem~\ref{basics}, we conclude that $\G^+$ is Menger.

To see  that also $\G$ is Menger
let us consider a sequence  $\la\K_n:n\in\w\ra$   of compact subsets of $\G^+$ and find
$\alpha$  such that $\K^\alpha=\K^\alpha_n$ for all $n\in\w$. It follows that
$\bigcup_{n\in\w}\K^\alpha_n\subset\la\bigcup\sfH_\alpha\ra_s$. Indeed, otherwise
by $(h)$ there exists $L_\alpha\in\bigcup_{n\in\w}\K^\alpha_n$ such that $\{\w\setminus L_\alpha\}\in\sfF_{\alpha+1}$,
and therefore $\w\setminus L_\alpha\in\G$, which contradicts $L_\alpha\in\G^+$. Thus
$\la\K^\alpha_n:n\in\w\ra$ fulfills the premises of $(g)$, which yields
 an increasing sequence $\la l_n:n\in\w\ra\in\w^\w$
with the property $\bigcup_{n\in\w}(K_n\cap l_n)\in\G^+$ for any
$\la K_n:n\in\w\ra\in\prod_{n\in\w}\K^\alpha_n.$ Applying
Theorem~\ref{basics}, we conclude that $\G$ is Menger.
\smallskip

Finally, we show how to construct a sequence as above
 satisfying $(a)$-$(j)$. Limit stages are straightforward, so
 suppose that
we have already constructed $\la\la \sfF_\beta,\sfH_\beta\ra:\beta\leq\alpha\ra$ satisfying $(e)$-$(j)$
for all $\beta<\alpha$, $(a),(b)$ and $(d)$ for all $\beta\leq\alpha$, and
$(c)$ for all $\beta\leq\beta'\leq\alpha$. Two cases are possible.
\smallskip

1). $\bigcup_{n\in\w}\K^\alpha_n \subset \la\bigcup\sfF_\alpha\ra$. Then
for every $\HH\in \sfH_\alpha$   we can find an increasing sequence $\la m^{\HH}_n:n\in\w\ra\in\w^\w$
such that $$ K\cap H\cap (m^{\HH}_n\setminus n) \neq\emptyset $$
 for every $K\in\la\K^\alpha_{n}\ra_n$,
$H\in \HH$, and $n\in\w$. Such an $m^{\HH}_n$ exists because of the compactness of the involved sets
and $\HH\subset\bigcup\sfH_\alpha\subset \la\bigcup\sfF_\alpha\ra^+$.
Since $|\sfH_\alpha|<\hot d$, there exists  an increasing sequence $\la m^\alpha_n:n\in\w\ra\in\w^\w$
such that $|\{n\in\w:m^{\HH}_n\leq m^\alpha_n\}|=\w$ for all $\HH\in \sfH_\alpha$.
Set
$$\K_\alpha=\Big\{\bigcup_{n\in\w}(K_n\cap m^\alpha_n):\la K_n:n\in\w\ra\in\prod_{n\in\w}\K^\alpha_n \Big\},$$
and $\sfF_{\alpha+1}$ to be the $\la\ra_*$-saturation of
$\sfF_\alpha \cup \{\K_\alpha\}.$
Since $|\sfF_{\alpha+1}|<\hot c$,
Lemma~\ref{bidi} yields  $Z_\alpha\subset\w$
 such that $\{\phi_\alpha^{-1}[Z_\alpha],\phi_\alpha^{-1}[\w\setminus Z_\alpha]\}\in \la\bigcup\sfF_{\alpha+1}\ra^+$.
 We define $\sfH_{\alpha+1}$ to be
 the $(\sfF_{\alpha+1},\wedge)$-saturation of
 $$\sfH_{\alpha}\cup\big\{ \{\phi_\alpha^{-1}[Z_\alpha],\w\setminus\phi_\alpha^{-1}[Z_\alpha]\}\big\}.$$
We are left with the task of checking that conditions $(a)$-$(j)$
are satisfied. Indeed, $(a)$, $(c)$, and $(e)$-$(j)$ hold
immediately by the construction, in case of $(f)$ and $(h)$ because
of the premises being violated.

Proof of $(b)$, $(d)$ for $\alpha+1$: \ By $(b)$ and $(d)$ for
$\alpha$ it suffices to prove that $\la\K_\alpha\ra_k\wedge
\HH\subset[\w]^\w$  for any
 $\HH\in\sfH_\alpha$ and $k\in\w$.
Let us fix
$$\big\{ \la K^i_n:n\in\w\ra :i\leq k\big\}\subset\prod_{n\in\w}\K^\alpha_n $$
  and   $H\in\HH$, and find
 $n\geq k $  such that $m^{\HH}_n\leq m^\alpha_n$.
 Then $\bigcap_{i\leq k}K^i_n\in\la\K^\alpha_n\ra_n$ and therefore
 \begin{eqnarray*}
 \emptyset\neq\bigcap_{i\leq k}K^i_n\cap H\cap (m^{\HH}_n\setminus n)\subset \bigcap_{i\leq k}K^i_n\cap H\cap (m^{\alpha}_n\setminus n)\subset \\
 \subset\Big(\bigcap_{i\leq k}\bigcup\{K^i_n\cap m^\alpha_n:n\in\w\}\cap H\Big)\setminus n,
 \end{eqnarray*}
 which proves $(b)$ and $(d)$ for $\alpha+1$ and thus completes this case.
 \smallskip

 2).  $\bigcup_{n\in\w}\K^\alpha_n \not\subset \la\bigcup\sfF_\alpha\ra$.
  Two subcases of  $2)$ are possible.

 $2_0)$.   $\bigcup_{n\in\w}\K^\alpha_n \not \subset \la\bigcup\sfH_\alpha\ra_s$.
It follows that there exists $L_\alpha\in
\bigcup_{n\in\w}\K^\alpha_n$ such that
 $\w\setminus L_\alpha\in\la\bigcup\sfH_{\alpha}\ra_s^+\subset\la\bigcup\sfF_\alpha\ra^+$,
 which allows us to define
 $\sfF_{\alpha+1}$ as the $\la\ra_*$-saturation of $\sfF_\alpha\cup\big\{ \{\w\setminus L_\alpha\}\big\} $.
 Lemma~\ref{bidi} yields  $Z_\alpha\subset\w$
 such that $\{\phi_\alpha^{-1}[Z_\alpha],\phi_\alpha^{-1}[\w\setminus Z_\alpha]\}\subset \la\bigcup\sfF_{\alpha+1}\ra^+$.
Let   $\sfH_{\alpha+1}$ be the $(\sfF_{\alpha+1},\wedge)$-saturation of
  $$\sfH_{\alpha}\cup\{\{\phi_\alpha^{-1}[Z_\alpha],\w\setminus\phi_\alpha^{-1}[Z_\alpha]\}\}.$$
 Conditions  $(a)$-$(c)$, and $(e)$-$(j)$ hold for $\alpha+1$
immediately by the construction, in case of $(e)$ and $(g)$ because of the premises being violated,
and for $(f)$ we can simply take $K_\alpha$ to be $L_\alpha$.
Regarding $(d)$ for $\alpha+1$, by $(b)$ and $(d)$ for $\alpha$ it suffices to prove that
 $F_0\cap F_1\cap(\w\setminus L_\alpha)\cap X$ is infinite for any $F_0,F_1\in\bigcup\sfF_\alpha$ and $X\in\{\phi_\alpha^{-1}[Z_\alpha],\w\setminus\phi_\alpha^{-1}[Z_\alpha]\}$, which again has been guaranteed in the course of the construction above.

$2_1)$. $\bigcup_{n\in\w}\K^\alpha_n  \subset \la\bigcup\sfH_\alpha\ra_s$.
In this case we set $\sfF_{\alpha+1}=\sfF_\alpha$.

For every $\F\in \sfF_{\alpha+1}$   we can find an increasing sequence $\la l^{\F}_n:n\in\w\ra\in\w^\w$
such that $$ K\cap F\cap (l^{\F}_n\setminus n) \neq\emptyset $$
 for every $K\in \K^\alpha_n$
and $F\in \F$. Such an $l^{\F}_n$ exists because of the compactness of the involved sets
and $\K^\alpha_n\subset\la\bigcup\sfH_\alpha\ra_s\subset \la\bigcup\sfF_\alpha\ra^+$.
Since $|\sfF_\alpha|<\hot d$, there exists  an increasing sequence $\la l^\alpha_n:n\in\w\ra\in\w^\w$
such that $|\{n\in\w:l^{\F}_n\leq l^\alpha_n\}|=\w$ for all $\F\in \sfF_{\alpha+1}$.
Set
$$\mathcal L_\alpha=\Big\{\bigcup_{n\in\w}(K_n\cap l^\alpha_n):\la K_n:n\in\w\ra\in\prod_{n\in\w}\K^\alpha_n \Big\}.$$
Lemma~\ref{bidi} yields  $Z_\alpha\subset\w$
 such that $\{\phi_\alpha^{-1}[Z_\alpha],\phi_\alpha^{-1}[\w\setminus Z_\alpha]\}\subset \la\bigcup\sfF_{\alpha+1}\ra^+$.
From $\bigcup_{n\in\w}\K^\alpha_n  \not\subset \la\bigcup\sfF_\alpha\ra$
it follows that
  there exists $K_\alpha\in \bigcup_{n\in\w}\K^\alpha_n$ such that
 $\{\w\setminus K_\alpha\}\in \la\bigcup\sfF_\alpha\ra^+$.
Finally, we define $\sfH_{\alpha+1}$ to be
 the $(\sfF_{\alpha+1},\wedge)$-saturation of
 $$\sfH_{\alpha}\cup \big\{ \{\phi_\alpha^{-1}[Z_\alpha], \{\w\setminus\phi_\alpha^{-1}[Z_\alpha],\w\setminus K_\alpha\}\big\}\cup \{\mathcal L_\alpha\}.$$
It is straightforward  to check that conditions $(a)$-$(j)$ are
satisfied, the only slightly non-trivial step  being $\mathcal
L_\alpha\subset\la\bigcup\sfF_{\alpha+1}\ra^+$, which can be proved
analogously  to (but more easily) ``$\la\K_\alpha\ra_k\wedge
\HH\subset[\w]^\w$ for any
 $\HH\in\sfH_\alpha$ and $k\in\w$'' in case $1).$

This concludes our proof, since all possible cases have been
considered.
\end{proof}

In the proof of the following corollary we shall use the classical
result of Hurewicz \cite{Hur27} (see also \cite[Theorem~4.3]{COC2})
stating that $X\subset\mathcal P(\w)$ is Menger if and only if
$f[X]$ is not dominating for any continuous $f:X\to\w^\w$.

\begin{corollary} \label{cor01}
If a semifilter $\HH$ satisfies condition $(2)$ of
Theorem~\ref{0main}, then $\HH^2$ is not Menger and $\HH$ is not
Scheepers. In particular, there is no continuous surjection from
$\HH$ onto $\HH^2$.
\end{corollary}
\begin{proof}
Suppose that $\HH^2$ is Menger, then so is $\HH\times(\sim\HH)$, where
$\sim\HH=\{\w\setminus H:H\in\HH\}$, because $\sim\HH$ is homeomorphic to $\HH$.
Therefore
$\mathcal X:=\HH\times(\sim\HH)\cap\{\la X,X\ra:X\subset\w\}$ is also Menger being a closed subspace
of  $\HH\times(\sim\HH)$, and hence so is  $\HH\cap(\sim\HH)$ as the projection of $\mathcal X$ to
the first (as well as the second) coordinate. Let us note that $\HH\cap(\sim\HH)$ consists of infinite co-infinite subsets of $\w$, and hence the map $h:\HH\cap(\sim\HH)\to\w^\w$,
$h(X)=\{n\in X: (n+1)\not\in X\}$, must have a non-dominating range.

On the other hand, given a strictly increasing $x\in\w^\w$ with $x(0)=0$, let us consider the
monotone surjection $\phi:\w\to\w$ such that $\phi^{-1}(n)=[x(n),x(n+1))$ for all $n$.
It follows that there exists $Z\subset\w$ such that $\phi^{-1}[Z]\in \HH\cap(\sim\HH)$,
and it is easy to see that $h[\phi^{-1}[Z]]\subset x[\w]$. Thus for every $x$ as above there exists
$X\in \HH\cap(\sim\HH)$ with $h(X)$ contained in the range of $x$, which clearly yields
a dominating continuous range of $\HH\cap(\sim\HH)$ and thus leads to a contradiction.

Now, suppose that $\HH$ is Scheepers and consider the clopen cover
$\mathsf O=\{\mathcal O_k:k\in\w\}$ of $\HH$, where $\mathcal
O_k=\{X\subset\w:k\in X\}$. By \cite[Theorem~2]{BabKocSch04} (see
the equivalence of items 1 and 4 there) there exists a disjoint
sequence $\la\mathsf O_n:n\in\w\ra$ of finite subsets of $\mathsf O$
such that for every finite $\HH'\subset \HH$ there exists $n$ with
$\HH'\subset\bigcup\mathsf O_n$. Let $s_n\in [\w]^{<\w}$ be such
that $\mathsf O_n=\{\mathcal O_k:k\in s_n\}$ and consider the
finite-to-one surjection $\phi:\w\to\w$ such that $\phi^{-1}(n)=s_n$
for all $n\in\w$. It follows from the above that for every finite
$\HH'\subset\HH$ there exists $n\in\w$ such that $H\cap
s_n\neq\emptyset$ for all $H\in\HH'$.

On the other hand, using Theorem~\ref{0main}(2) pick $Z\subset\w$ be such that $\HH':=\{\phi^{-1}[Z],\w\setminus\phi^{-1}[Z]\}\subset\HH$
and note that there is no $n\in\w$ such that both sets
$$ \phi^{-1}[Z] \cap s_n=\phi^{-1}[Z]\cap\phi^{-1}(n) \mbox{ \ and \ } (\w\setminus\phi^{-1}[Z])\cap s_n= (\w\setminus\phi^{-1}[Z])\cap \phi^{-1}(n)$$
are non-empty, a contradiction.
\end{proof}

Let us note that the proof of Corollary~\ref{cor01} could be
actually extracted from that of \cite[Lemma~3.1]{BelTokZdo16}, but
we have nonetheless presented it for reader's convenience. Since
every Menger subspace of $\mathcal P(\w)$ has Menger square and is
Scheepers in the Miller model, Theorem~\ref{0main} cannot be proved
in ZFC, as follows from Corollary~\ref{cor01}.

Since each filter has a structure of a topological group,
Theorem~\ref{0main} combined with Corollary~\ref{cor01} answers
\cite[Question~1.8]{BelTokZdo16} in the affirmative, the coideal
$\mathcal G^+$ being the needed counterexample. Another motivation
for Corollary~\ref{cor01} comes from \cite{MedZdo16} where it was
proved that each filter on $\w$ is homeomorphic to its square.
According to \cite[Prop.~8]{MedZdo16}, this result fails for
semifilters, and the counterexample is a Borel comeager semifilter.
However, until now no  semifilter
$\F$ such that both $\F$ and $\F^+$ are non-meager, which in addition  is not homeomorphic to its square, was known, and
Theorem~\ref{0main} combined with Corollary~\ref{cor01}  gives a
consistent example of coideal like that.

For curiosity we exclude below one more possibility for spaces
related to filters $\G$ satisfying Theorem~\ref{0main} to be
homeomorphic.

\begin{corollary} \label{cor02}
If a filter  $\G$ satisfies  Theorem~\ref{0main}, then $\G\times\G^+$ is not homeomorphic to
 to $\G$.
\end{corollary}
\begin{proof}
 $\G^2$ is Menger by \cite{MedZdo16} (see also \cite[Claim~5.5]{ChoRepZdo15} for a simpler proof),
whereas $(\G\times\G^+)^2$ can be mapped continuously onto $(\G^+)^2$ and hence is not Menger.
\end{proof}

\section{Examples by Forcing}

This section was inspired by \cite[\S~4]{GuzHruMar??} as well as
\cite[\S~6]{Laf89}. More precisely, one can obtain a filter $\G$
such as in Theorem~\ref{0main} by countably complete forcing, namely
let $\IP$ be the poset consisting of conditions $p=\la \mathsf
F,\mathsf H\ra$ such that
\begin{itemize}
\item[$(i)$] $\mathsf F$ is a countable collection of compact subsets of $[\w]^\w$ such that
$\bigcup\mathsf F$ is centered;
\item[$(ii)$] $\mathsf H$ is a countable collection of
compact subsets of $[\w]^\w$ such that $\mathsf F\subset\mathsf H$;
\item[$(iii)$] $\bigcup\mathsf H\subset \la\bigcup\mathsf F\ra^+$.
\end{itemize}
A condition $\la \sfF_1,\sfH_1\ra$ is stronger than $\la
\sfF_0,\sfH_0\ra$ (and written $\la \sfF_1,\sfH_1\ra \leq \la
\sfF_0,\sfH_0\ra$) if $\sfF_1\supset\sfF_0$ and
$\sfH_1\supset\sfH_0$.

\begin{theorem} \label{0main_f}
Let $G$ be a $\IP$-generic filter, $\G=\bigcup\{\bigcup\sfF:\exists\sfH\: (\la\sfF,\sfH\ra\in G)\}$,
and $\HH=\bigcup\{\bigcup\sfH:\exists\sfF\: (\la\sfF,\sfH\ra\in G)\}$.
Then $\G$ is a filter, $\HH=\G^+$, and both $\G$ and $\HH$ are Menger.
Moreover, for every  surjection $\phi:\w\to\w$ there exists $X\subset\w$ such that
$\phi^{-1}[X],\phi^{-1}[\w\setminus X]\in\G^+$.
\end{theorem}

We leave the proof of Theorem~\ref{0main_f} to the reader, as it is
more or less  a kind of a repetition of that of Theorem~\ref{0main},
with the only difference being that now everything we need would
happen ``generically'', i.e., the set of suitable conditions is
dense, while for Theorem~\ref{0main} we had to ``manually''
guarantee all that by going over appropriate enumerations.

Instead, we shall address a similar poset tailored to analyze the
density one filter  $\mathcal Z^*$ on $\w$ consisting of those
$Z\subset\w$ such that $\lim_{n\to\infty}\frac{|Z\cap n|}{n}=1$. It
is a well known open problem   whether there exists a proper poset
adding no dominating reals but adding an infinite subset of $\w$
almost included into all ground model elements of $\mathcal Z^*$,
see, e.g., \cite[Question~2.12]{Hru11}. This  motivated us to
introduce the following  poset.

Let $\IQ$ be the set of conditions
$\la\mathsf F,\HH, \ve\ra$
such that
\begin{itemize}
\item[$(i)$] $\mathsf F$ is a countable collection of compact subsets of $[\w]^\w$ such that
$\bigcup\mathsf F$ is centered;
\item[$(ii)$] $\HH$ is a countable subset of $[\w]^\w$ and
$\ve:\HH\to (0,1]$; and
\item[$(iii)$] For every $F\in\la\sfF\ra$ and $H\in\HH$ there exists $X\in [\w]^\w$ such that
$\lim_{n\in X}\frac{|F\cap n|}{n}=1$
and $\liminf_{n\in X}\frac{|F\cap H\cap n|}{n}\geq\ve(H)$.
\end{itemize}
A condition $\la\mathsf F_1,\HH_1,\ve_1\ra$ is stronger than
$\la\mathsf F_0,\HH_0,\ve_0\ra$ (and written $\la\mathsf
F_1,\HH_1,\ve_1\ra\leq  \la\mathsf F_0,\HH_0,\ve_0\ra$) if
$\sfF_1\supset\sfF_0$,  $\HH_1\supset\HH_0$, and
$\ve_0=\ve_1\uhr\HH_0$. Clearly, $\IQ$ is countably closed.

\begin{theorem} \label{0main_density_z}
Let $G$ be a $\IQ$-generic filter and $\G=\bigcup\{ \bigcup \sfF\: :\: \exists \HH\exists\ve\:(\la\sfF,\HH,\ve\ra\in G)\}$.
Then $\G$ is a filter, $\mathcal Z^*\subset\G$, and $\G^+$ is Menger.
\end{theorem}
\begin{proof}
To see that $\G$ is a filter note that for any $\la\sfF,\HH,\ve\ra\in\IQ$ and finite $\mathcal F\subset\la \bigcup\sfF\ra$
we have $\la \sfF\cup\{ \{\bigcap \mathcal F\}\},\HH,\ve\ra\in \IQ$, and hence the set of all conditions
in $\IQ$ whose first component
contains  $\{\bigcap\mathcal F\}$ as an element, is dense below $\la\sfF,\HH,\ve\ra$.

The fact that $\Z^*\subset\G$, follows from the observation that for
any $T\in\mathcal Z^*$ and $\la\sfF,\HH,\ve\ra\in\IQ$, we have that
$\la\sfF\cup\{\{T\}\},\HH,\ve\ra\in\IQ$. Indeed, if for some
 $F\in\la\sfF\ra$, $H\in\HH$, and $X\in [\w]^\w$ we have
$\lim_{n\in X}|\frac{|F\cap n|}{n}|=1$
and $\liminf_{n\in X}|\frac{|F\cap H\cap n|}{n}|\geq\ve(H)$, then
also $\lim_{n\in X}|\frac{|(F\cap T)\cap n|}{n}|=1$
and $\liminf_{n\in X}|\frac{|(F\cap T)\cap H\cap n|}{n}|\geq\ve(H)$.

Thus we are left with the task of showing that for every sequence
$\la \K_n:n\in\w\ra$ of compact subspaces of $\G$ there exists an
increasing sequence $\la m_n:n\in\w\ra\in\w^\w$ such that for every
$\la K_n:n\in\w\ra\in\prod_{n\in\w}{\K_n}$ we have
$\bigcup_{n\in\w}(K_n\cap m_n)\in\G$. Let
$\la\sfF_0,\HH_0,\ve_0\ra\in G$ be such that
$\la\sfF_0,\HH_0,\ve_0\ra\forces\bigcup_{n\in\w}\K_n\subset\G$. We
claim that there exists $\la\sfF_1,\HH_1,\ve_1\ra\in G$ such that
$\bigcup_{n\in\w}\K_n\subset\la\bigcup\sfF_1\ra$. Indeed, given any
condition $\la\sfF ,\HH,\ve\ra\leq\la\sfF_0,\HH_0,\ve_0\ra$ two
cases are possible.

1. There exists $K\in \la\bigcup_{n\in\w}\K_n\ra$ such that $\w\setminus K\in\HH$,
or there exists  $\delta>0$ such that
  for every $F\in\la\sfF\ra$ there exists $X\in [\w]^\w$ with properties
$\lim_{n\in X}\frac{|F\cap n|}{n}=1$
and $\liminf_{n\in X}\frac{|F\cap (\w\setminus K)\cap n|}{n}\geq\delta$.
In this case
$ \la\sfF,\HH\cup\{\w\setminus K\},\ve'\ra\forces K\not\in\G$, where
$\ve'\uhr\HH=\ve$ and $\ve'(\w\setminus K)=\delta$ if $\w\setminus K\not\in\HH$. The latter leads to a contradiction.

2.  For every $K\in \la\bigcup_{n\in\w}\K_n\ra$ we have $\w\setminus K\not\in\HH$,
and for every  $\delta>0$ there exists
$F\in\la\sfF\ra$ such that for every $X\in [\w]^\w$ with $\lim_{n\in X}\frac{|F\cap n|}{n}=1$ the following holds:
\begin{equation} \label{eq11}
\forall^* n\in X\: \frac{|F\cap (\w\setminus K)\cap n|}{n}<\delta.
\end{equation}
Let us note that if for some $\delta>0$ an element $F\in\la\sfF\ra$ is a witness for Equation~\ref{eq11},
then any smaller $F'\in\la\sfF\ra$ is also one.
Given any $F\in\la\sfF\ra$ and $K\in\la\bigcup_{n\in\w}\K_n\ra$, let us construct a
decreasing sequence $\la F_i:i\geq 1\ra$ of elements of
$\la\sfF\ra$ such that $F_0=F$ and
\begin{equation} \label{eq12}
\forall X\in [\w]^\w\: \Big(\lim_{n\in X}\frac{|F_i\cap n|}{n}=1 \Rightarrow
\forall^* n\in X\: \frac{|F_i\cap (\w\setminus K)\cap n|}{n}<\frac{1}{i}\Big).
\end{equation}
Equation~\ref{eq12} implies
\begin{equation} \label{eq13}
\forall X\in [\w]^\w\: \Big(\lim_{n\in X}\frac{|F_i\cap n|}{n}=1 \Rightarrow
\forall^* n\in X\: \frac{|F_i\cap K\cap n|}{n}\geq 1-\frac{2}{i}\Big).
\end{equation}
Let us fix now any $H\in\HH$ and
 for every $i$ find $X_i\in [\w]^\w$
such that $\lim_{n\in X_i}\frac{|F_i\cap n|}{n}=1$
and $\frac{|F_i\cap H\cap n|}{n}>\ve(H)-\frac{1}{i}$ for all $n\in X_i$.
This is possible by item $(iii)$ of the definition of $\IQ$.
Removing finitely many elements of $X_i$, if necessary, by Equation~\ref{eq13} we may assume that
$\frac{|F_i\cap K\cap  n|}{n}>1-\frac{2}{i}$  for all $n\in X_i$.
Now let $\la n_i:i\in\w\ra\in\w^\w$ be an increasing sequence
such that $n_i\in X_i$ and $X(F,H,K)=\{n_i:i\in\w\}$.
It follows that
$\frac{|F_i\cap H\cap n_i|}{n_i}>\ve(H)-\frac{1}{i}$ and $\frac{|F_i\cap K\cap  n_i|}{n_i}>2-1/i$  for all $i\in \w$.
Since $F_i\subset F$ for all $i$, we have that
$\frac{|F\cap H\cap n_i|}{n_i}>\ve(H)-\frac{1}{i}$ and $\frac{|F\cap K\cap  n_i|}{n_i}>1-2/i$
and therefore $\frac{F\cap K\cap H\cap n_i}{n_i}>\ve(H)-\frac{3}{i}$
  for all $i\in \w$.
Thus $\la\sfF\cup\{\K_n:n\in\w\},\HH,\ve\ra\in \IQ$, where for each
$F\in\la \sfF\ra$, $H\in\HH$, and $K\in\la\bigcup_{n\in\w\K_n}\ra$ the infinite set $X(F,H,K)$
is such as required in $(iii)$ for $F\cap K$ and $H$.

Summarizing, we have proved that for any
$\la\sfF ,\HH,\ve\ra\leq\la\sfF_0,\HH_0,\ve_0\ra$
we have $\la\sfF\cup\{\K_n:n\in\w\},\HH,\ve\ra\in \IQ$, and thus there exists
$\la\sfF_1,\HH_1,\ve_1\ra\in G$ with $\bigcup_{n\in\w}\K_n\subset\la\bigcup\sfF_1\ra$.

Let us now fix $\la\sfF ,\HH,\ve\ra\leq\la\sfF_1,\HH_1,\ve_1\ra$ and  enumerations
$\la H_i:i\in\w \ra$ of $\HH$ as well as $\la\F_i:i\in\w\ra$ of $\sfF$.
Set
$$\mathcal L_i=\big\{\bigcap\mathcal Y\: :\: \mathcal Y\in [\bigcup_{j\leq i}\F_j\cup\bigcup_{j\leq i}\K_j]^{\leq i}\big\}$$
and by recursion over $i$ construct an increasing number sequence
$\la m_i:i\in\w\ra$ such that for every $L\in\mathcal L_i$ and $j\leq i$ there exists
$n_j\in [m_{i-1},m_{i})$ such that
$\frac{|L\cap n_j|}{n_j}>1-\frac{1}{i}$ and $\frac{|L\cap H_j\cap n_j|}{n_j}>\ve(H_j)(1-\frac{1}{i})$.
This is possible by $(iii)$ and the compactness of $\mathcal L_i$.
Letting
$$\K=\Big\{\bigcup_{i\in\w}(K_i\cap m_i)\: :\:
\la K_i:i\in\w\ra\in\prod_{i\in\w}{\K_i}\Big\},$$
 we claim that
 $\la \sfF\cup\{\K\},\HH,\ve\ra\in\IQ$. Indeed, let us fix $k\in\w$ and a family
 $\{Y_s:s\leq k\}\subset\K$, where
 $Y_s=\bigcup_{i\in\w}(K_{i,s}\cap m_i)$ for some $K_{i,s}\in\K_i$.
 Also, let us fix a family  $\{R_s:s\leq k\}\subset\bigcup_{s\leq k}\F_s$
 and set $R=\bigcap_{s\leq k}R_s$. Then for every
 $i\geq 2k+1$ and $s\leq k$ we  have that
 $$\mathcal Y:=\{R_s:s\leq k\}\cup\{K_{i,s}:s\leq k\}\in [\bigcup_{j\leq i}\F_j\cup\bigcup_{j\leq i}\K_j]^{\leq i},$$
and therefore
for every and $j\leq i$ there exists
$n_j\in [m_{i-1},m_{i})$ such that
$\frac{|\bigcap\mathcal Y\cap n_j|}{n_j}>1-\frac{1}{i}$ and $\frac{|\mathcal Y\cap H_j\cap n_j|}{n_j}>\ve(H_j)(1-\frac{1}{i})$.
Since
\begin{equation} \label{eq_14}
 \bigcap\mathcal Y\cap n_j\subset\bigcap\{R_s:s\leq k\}\cap \bigcap\{Y_s:s\leq k\} \cap n_j,
 \end{equation}
(because $K_{i,s}\cap m_i\subset Y_s\cap m_i$ for all $s\leq k$), we
conclude that $(iii)$ is satisfied for $\la
\sfF\cup\{\K\},\HH,\ve\ra$, and hence it is a member of $\IQ$. Thus
for arbitrary $\la\sfF ,\HH,\ve\ra\leq\la\sfF_1,\HH_1,\ve_1\ra$ we
have $\la\sfF\cup\{\K\} ,\HH,\ve\ra\in\IQ$, and hence there exists
$\la\sfF_2 ,\HH_2,\ve_2\ra\in G$ with $\K\in\sfF_2$, which yields
$\K\subset\G$. It remains to apply Theorem~\ref{basics}.
\end{proof}

It is well-known and easy to see that the filter $\mathcal Z^*$ cannot be extended to any $P^+$-filter, and hence 
$\G$ from Theorem~\ref{0main_density_z} is not Menger by Corollary~\ref{obs_obv}. In addition,
the Menger co-ideal
$\G^+$  does not contain any $P$-point, and hence also no Menger ultrafilter.
We do not know whether such examples can be obtained in ZFC, see Section~\ref{problem}
for more  questions related to $\mathcal Z$.

\section{Impossibility Results} \label{imp_res}

Here we show that Theorem~\ref{0main} cannot be proved without additional set-theoretic assumptions,
see Theorem~\ref{0_anti_main} below. Following \cite{BelTokZdo16}
for a semifilter $\F$ we denote by $\IP_\F$ the poset consisting
of
all partial maps $p$ from $\w\times\w$ to $2$ such that for every
$n\in\w$ the domain of $p_n:k\mapsto p(n,k)$ is an element of
$\sim\F=\{\w\setminus F:F\in\F\}$. If, moreover, we assume that additionally
$\mathrm{dom}(p_n)\subset
\mathrm{dom}(p_{n+1})$ for all $n$, the corresponding poset will
be
denoted by $\IP_\F^*$. A condition $q$ is stronger than $p$ (in
this
case we write $q\leq p$) if $p\subset q$. For filters $\F$ the
poset
$\IP_\F^*$ is obviously dense in $\IP_\F$, and the latter is
proper
and $\w^\w$-bounding if $\F$ is a non-meager $P$-filter
\cite[Fact
VI.4.3, Lemma~VI.4.4]{She_propimp}.
This result has the following topological counterpart
proved  in \cite{BelTokZdo16}. Recall that a poset $\IP$ is called
\emph{$\w^\w$-bounding} if $(\w^\w)^V$ is dominating in $V^{\IP}$.

\begin{lemma} \label{w^w-bound}
If $\F^+$ is a Menger semifilter, then both  $\IP_\F$ and
$\IP_\F^*$
are proper and $\w^\w$-bounding.
\end{lemma}

We shall need the following game of length $\omega$ on a topological space $X$:
 In the $n$th  move player $I$ chooses an  open cover
$\U_n$ of $X$, and player $\mathit{II}$ responds by choosing a
finite $\V_n\subset \U_n$. Player $\mathit{II}$ wins the game if $\bigcup_{n\in\omega}
\bigcup\V_n =X$. Otherwise, player $I$ wins.  We shall call this game
 \emph{the Menger game} on $X$.
 It is well-known  that
  $X$  is Menger  if and only if
 player $I$ has no winning strategy in the Menger game on
$X$, see \cite{Hur25} or \cite[Theorem~13]{COC1}.

For a relation $R$ on $\w$ and $x,y\in\w^\w$ we denote by
$[x\:R\:y]$ the set $\{n: x(n)\:R\:y(n)\}$.
The next lemma improves \cite[Lemma~4.3]{BelTokZdo16}.

\begin{lemma} \label{meng_forcing}
Suppose that $\F$ is a semifilter  such that $\F\subset\F^+$ and $\F^+$ is Menger.
 Let
$x$ be   $\IP_\F^*$-generic, $\IQ\in V[x]$ be an
$\w^\w$-bounding poset, and $H$  a $\IQ$-generic over $V[x]$.
Then
in $V[x*H]$ there is no  Menger semifilter $\G$ containing $\F$
such that $\G\subset\G^+$ and $\G^+$ is Menger.
\end{lemma}
\begin{proof}
Suppose that a semifilter $\G\in V[x*H]$ as above exists. Throughout
the proof we shall identify $x$ with $\bigcup x:\w\times\w\to 2$.
Suppose to the contrary, that such a $\G$ exists. Set
$x_j(n)=x(j,n)$. Since $\G\subset\G^+$, $\G$ cannot contain two
disjoint elements, and hence for every $j$ there exists $\ve_j\in 2
$ such that $X_j:=x_j^{-1}(\ve_j)\in\G^+$. Without loss of
generality we may assume that there exists an infinite $T\subset\w$
such that $\ve_j=1$ for all $j\in T$. Since $\IP_\F^* *\IQ$ is
$\w^\w$-bounding we can get an increasing sequence  $\la j_k:k\in\w
\ra\in V$  such that for every $k\in\w$ there exists $t_k\in
[j_k,j_{k+1})\cap T$.

Given $k\in\w$,  for every   $s\in\w^{j_{k+1}}$ set
$$\U_{k,s}=\{X\subset\w\: :\: s[j_{k+1}]\subset X \}.$$
Then $\{\U_{k,s}:s\in\Sigma_{k,l}\}$, where
$\Sigma_{k,l}=\prod_{j<j_{k+1}}(X_j\setminus l)$,
 is an open cover of
$\G$ for every $k,l\in\w$, because each  $Y\in \G$ has infinite
intersection with each $X_j$. We shall define
 a strategy $\Theta$ for  player $I$ in the Menger game on $\G$ as follows.
 $\Theta(\emptyset)=\{\U_{0,s}:s\in \Sigma_{0,0}\}$.
 If $\mathit{II}$ replied with some
 finite subset $\mathcal V_0$ of  $\Theta(\emptyset)$, $I$ finds $h(0)\in\w$ such that
 this reply is contained in $\{\U_{0,s}:s\in h(0)^{j_1}\cap \Sigma_{0,0}\}$ and
 his next move is then
$\Theta\la\mathcal V_0\ra:= \{\U_{1,s}:s\in \Sigma_{1, h(0)}\}.$
And so on, i.e., after $k$ many rounds
player $I$ finds $h(k)\in\w$ above $h(k-1)$ (here $h(-1):=0$) such that the reply of $\mathit{II}$
is contained in $\{\U_{k,s}:s\in h(k)^{j_{k+1}}\cap \Sigma_{k,h(k-1)}\}$ and
 his next move is then
$$\Theta\la\mathcal V_0,\ldots,\mathcal V_k\ra:= \{\U_{k+1,s}:s\in \Sigma_{k+1,h(k)}\}.$$
Since $\G$ is Menger,
 $\Theta$ is not winning, and hence there exists
an increasing sequence $h\in\w^\w\in V[x*H]$
such that
\begin{equation} \label{eq_new_1}
 \G\subset\bigcup\big\{\bigcup\{\U_{k,s}: s\in h(k)^{j_{k+1}}\cap \Sigma_{k,h(k-1)} \} \: :\: k\in\w           \big\}.
 \end{equation}

 Next, we shall define
 a strategy $\Upsilon$ for  player $I$ in the Menger game on $\G^+$ as follows.
 $\Upsilon(\emptyset)=\{\U_{0,s}:s\in S_{0,0}\}$, where
 $S_{k,l}=\prod_{j<j_{k+1}}\big(\w\setminus(\mathrm{dom}(p_j)\cup l)\big)$ for all $k,l\in\w$.
 If $\mathit{II}$ replied with some
 finite subset $\mathcal V_0$ of  $\Upsilon(\emptyset)$, $I$ finds $l_0\in\w$ such that
 this reply is contained in $\{\U_{0,s}:s\in h(l_0)^{j_1}\}$ and
 his next move is then
$\Upsilon(\mathcal V_0):= \{\U_{1,s}:s\in S_{1,h(l_0)}\}.$
And so on, i.e., after $k$ many rounds
player $I$ finds $l_k\in\w$ above $l_{k-1}$ such that the reply of $\mathit{II}$
is contained in $\{\U_{k,s}:s\in h_1(l_k)^{j_{k+1}}\}$ and
 his next move is then
$$\Upsilon\la\mathcal V_0,\ldots,\mathcal V_k\ra:= \{\U_{k+1,s}:s\in S_{k+1,h(l_k)}\}.$$
Since $\Upsilon$ is not winning, there exists
an increasing sequence $\la l_k:k\in\w\ra$ (we set also $l_{-1}=0$ for convenience)
such that
\begin{equation} \label{eq_new_2}
 \G^+\subset\bigcup_{k\in\w}\bigcup\big\{\U_{k,s}: s\in h(l_k)^{j_{k+1}}\cap S_{k, h(l_{k-1})}\big\}.
 \end{equation}
Let $h_1\in V\cap \w^\w$ be an increasing function such that $h_1(k)\geq h(k)$ for all $k\in\w$,
and
 $\la p,\dot{q}\ra\in x*H$  a condition forcing all the above to happen.
 Let also $\name{h},$ $\name{t}_k,$ $\name{\G},$ $\name{\Sigma}_{k,l},$ $\name{x}$, $\name{x}_j$,
 $\name{X}_j$, $\name{T}$, ... be
 $\IP_{\F}^* *\IQ$-names of the objects in $V[x*H]$ considered above.
Consider the condition $p^1\in\IP_{\F}^*$ below $p$ defined as follows:
$$ p^1_j=p_j\cup\{\la n,0\ra: n\in h_1(l_k)\setminus \mathrm{dom}(p_j)\}$$
whenever $j\in [j_k,j_{k+1})$ and $k\in\w$. Thus $\la
p^1,\name{q}\ra$ forces Equations~\ref{eq_new_1} and \ref{eq_new_2}.
Thus $\la p^1,\name{q}\ra$ forces the following: If for every
$k\in\w$ and $s\in \name{h}(k)^{j_{k+1}}\cap
\Sigma_{k,\name{h}(k-1)}$ we pick $j_s\in j_{k+1}$, then
\begin{equation} \label{eq_new_3}
\big\{ s(j_s)\: :\: s\in \name{h}(k)^{j_{k+1}}\cap \Sigma_{k,\name{h}(k-1)}, k\in\w\big\}\in\name{\G}^+; \mbox{\ and}
\end{equation}
if for every $k\in\w$ and $s\in  \name{h}(\name{l}_k)^{j_{k+1}}\cap \name{S}_{k, \name{h}(\name{l}_{k-1})}$ we pick
$j'_s\in j_{k+1}$, then
\begin{equation} \label{eq_new_4}
\big\{ s(j'_s)\: :\: s\in \name{h}(\name{l}_k)^{j_{k+1}}\cap \name{S}_{k, \name{h}(\name{l}_{k-1})}, k\in\w\big\}\in\name{\G}.
\end{equation}
Given a $\IP_\F^* *\IQ$-generic filter $ x^1*H^1$  containing $\la
p^1,\name{q}\ra$, we shall work in $V[x^1*H^1]$ in what follows. For
abuse of notation we shall again use notations for the evaluations
with respect to $x^1*H^1$ of all names considered above, obtained
simply by removing ``$\name{\ }$'', i.e., $h:=\name{h}^{x^1*H^1}$
etc. This way letting $j'_{s}:=t_k$ for each   $s\in
h(l_k)^{j_{k+1}}\cap S_{k, h(l_{k-1})}$, the set $B'$
 defined in
Equation~\ref{eq_new_4} belongs to $\G$. For every $k\in\w$ and
$s\in h(k)^{j_{k+1}}\cap \Sigma_{k,h(k-1)}$ set $j_s:=t_{m(k)}$,
where $m(k)=\min\{m:l_m\geq k\}$. The set $B$
 defined in
Equation~\ref{eq_new_3} for this choice of $j_s$'s belongs to
$\G^+$, and hence $B\cap B'\neq\emptyset$. Thus there exist
$k,k'\in\w$, $s'\in h(l_{k'})^{j_{k'+1}}\cap S_{k', h(l_{k'-1})}$,
and $s\in h(k)^{j_{k+1}}\cap \Sigma_{k,h(k-1)}$ such that
\begin{equation}\label{eq_new_5}
s(t_{m(k)})=s'(t_{k'}).
\end{equation}
Since $s(t_{m(k)})\in h(k)\setminus h(k-1)$ and $s'(t_{k'})\in
h(l_{k'})\setminus h(l_{k'-1})$, from Equation~\ref{eq_new_5} we
conclude that $h(k)\setminus h(k-1)\subset h(l_{k'})\setminus
h(l_{k'-1})$, and hence $k'=\min\{m:l_m\geq k\}$, which yields
$k'=m(k)$. Recall that $s(t_{m(k)})\in X_{t_{m(k)}}$, and therefore
\begin{equation}\label{eq_new_6}
x_{t_{m(k)}}\big(s(t_{m(k)})\big)=\varepsilon_{t_{m(k)}}=1.
\end{equation}
On the other hand,
\begin{equation}\label{eq_new_7}
x_{t_{m(k)}}\big(s(t_{m(k)})\big)=
x_{t_{m(k)}}\big(s'(t_{m(k)})\big)=
p^1_{t_{m(k)}}\big(s'(t_{m(k)})\big)
\end{equation}
 because $p^1_{t_{m(k)}}\subset
x_{t_{m(k)}}$ and
$$s'(t_{m(k)})\in h(l_{m(k)})\subset
h_1(l_{m(k)})\subset\mathrm{dom}(p^1_{t_{m(k)}}),$$ the last
inclusion following from $t_{m(k)}\in [j_{m(k)}, j_{m(k)+1})$ and
the definition of $p^1$. Recall that
$s'(t_{m(k)})\not\in\mathrm{dom}(p_{t_{m(k)}})$ by the definition of
$S_{k',h(l_{k'-1})}=S_{m(k),h(l_{m(k)-1})}$, which implies that
$p^1_{t_{m(k)}}\big(s'(t_{m(k)})\big)=0$ by the definition of $p^1$.
Thus
 $x_{t_{m(k)}}\big(s(t_{m(k)})\big)=
x_{t_{m(k)}}\big(s'(t_{m(k)})\big)=0$ by Equation~\ref{eq_new_7},
which is impossible by Equation~\ref{eq_new_6}. This contradiction
completes our proof.
 \end{proof}
The next theorem improves \cite[Theorem~4.5]{BelTokZdo16} and is
the main result of this section.
\begin{theorem} \label{0_anti_main}
It is consistent that there is no semifilter $\G$ on $\w$ such that
$\G\subset\G^+$ and both $\G,\G^+$ are Menger.
\end{theorem}
\begin{proof}
Let us assume that $V=L$  and\footnote{It suffices to assume
$2^\w=\w_1$, $2^{\w_1}=\w_2$, and
$\diamond_{\{\delta\in\w_2:\mathrm{cf}(\delta)=\w_1\}}$, see
\cite[Theorem~5.13]{She_propimp}.} consider a function
$B:\w_2\to H(\w_2)$, the family of all sets whose transitive closure
has size $<\w_2$. Let $\la
\IP_\alpha,\name{\IQ}_\beta:\beta<\alpha\leq\w_2\ra$ be the
following  iteration with at most countable supports: If $B(\alpha)$
is a $\IP_\alpha$-name for $\IP^*_{\name{\F}}$ for some semifilter
$\name{\F}$ such that $\forces_{\IP_\alpha}$
``$\name{\F}\subset\name{\F}^+$ and $\name{\F},\name{\F}^+$ are
Menger'', then $\name{\IQ}_\alpha=B(\alpha)$. Otherwise we let
$\name{\IQ}_\alpha$  be a $\IP_\alpha$-name for the trivial forcing.
Then $\IP_{\w_2}$ is $\w^\w$-bounding forcing notion with
$\w_2$-c.c. being a countable support iteration of length $\w_2$ of
proper $\w^\w$-bounding posets of size $\w_1$  over a model of GCH.
Now, using the suitable diamond in $V$ for the choice of $B$
together with a standard reflection argument, we can guarantee in
addition that for any $\IP_{\w_2}$-generic filter $G$ over $V$ and
semifilter $\F\in V[G]$ such that $\F\subset\F^+$ and $\F,\F^+ $ are
Menger, the following holds:
\begin{quote}
The set $\big\{\alpha:\F_\alpha:=(\F\cap V[G\cap\IP_\alpha])\in
V[G\cap\IP_\alpha], $ $\F_\alpha\subset\F^+_\alpha$,
$\F_\alpha,\F^+_\alpha$ are Menger in $V[G\cap\IP_\alpha],$ and
$\name{\IQ}_\alpha^{G\cap\IP_\alpha}=\IP^*_{\name{\F}_\alpha}\big\}$
is stationary in $\w_2$.
\end{quote}
 Now, a direct application of
Lemma~\ref{meng_forcing} implies that $\F_\alpha$ cannot be enlarged
to any semifilter $\U\subset\U^+$ in $V[G]$ such that $\U,\U^+$ are
Menger,  which contradicts the fact that  $\F$ is such an
enlargement.
\end{proof}

\section{Open Questions} \label{problem}

Theorem~\ref{0main} together with Corollaries~\ref{cor01} and \ref{cor02} motivate the following

\begin{question}\,
\begin{itemize}
\item Is there a ZFC example of a non-meager filter $\F$ such that $\F^+$ is not homeomorphic to
$(\F^+)^2$?
\item Is $\F\times\F^+$ homeomorphic to $\F^+$ for every non-meager filter $\F$?
\end{itemize}
\end{question}

In light of Theorem~\ref{0main_density_z} it is natural to ask the next question.
Let us recall from \cite{ChoRepZdo15} that for a filter $\F$ the Mathias forcing associated to it
adds no dominating reals iff $\F$ is Menger, so the following question is
especially interesting when $\F$ is not an ultrafilter.

\begin{question}
Let $\F$ be a filter such that $\F^+$ is Menger. Is there a (proper) poset
adding no dominating reals and adding an infinite pseudointersection of $\F$?
What about the Mathias forcing associated to $\F^+$?
\end{question}

Theorem~\ref{0_anti_main} leaves open one possibility whose inconsistency we are unable to establish.

\begin{question}\label{main_q}\,
\begin{itemize}
\item Is there a ZFC example of a filter $\F$ such that $\F^+$ is Menger?
\item Is there a ZFC example of a semifilter $\F\subset\F^+$ such that $\F^+$ is Menger?
\end{itemize}
\end{question}

The first item of the question above  has been mentioned to us by Miko\l aj Krupski in private communication and
is related to his work \cite{Kru??}.
By \cite[Lemma~3.1]{BelTokZdo16} the negative answer to the second item of Question~\ref{main_q} implies that consistently
every  Menger remainder of a topological group is $\sigma$-compact.

\end{document}